\documentclass[12pt,a4paper,oneside,reqno]{article}

\usepackage[utf8]{inputenc}
\usepackage[english]{babel}
\usepackage[left=1in,right=1in,top=1in,bottom=1in]{geometry}
\usepackage{amsfonts,amssymb,amsmath,amsthm}
\usepackage[symbol]{footmisc}
\usepackage{longtable,array}
\newcolumntype{L}[1]{>{\raggedright\arraybackslash$}p{#1}<{$}}
\usepackage{enumerate}
\usepackage{cite}
\usepackage{xcolor}
\newcommand{\C}{\mathbb{C}}
\usepackage{hyperref}
\hypersetup{hidelinks}
\usepackage{fancyhdr}
\theoremstyle{plain}
\newtheorem{theorem}{Theorem}
\newtheorem{lemma}[theorem]{Lemma}

\newtheorem{definition}[theorem]{Definition}
\newtheorem*{lemmaA}{Lemma A}

\numberwithin{theorem}{section}
\numberwithin{equation}{section}

\DeclareMathOperator{\Der}{Der}
\DeclareMathOperator{\Ann}{Ann}

\IfFileExists{fouriernc.sty}{\usepackage{fouriernc}}{}

\newcommand{\doi}[1]{\href{https://doi.org/#1}{doi:#1}}
\newcommand{\arxiv}[1]{\href{https://arxiv.org/abs/#1}{arXiv:#1}}
\newcommand{\arxivsearch}[1]{\href{https://arxiv.org/search/?query=#1}{arXiv}}

\newcommand{\F}{\mathbb{F}}\newcommand{\J}{\mathcal{J}}
\newcommand{\bB}[2]{\mathbf{B}_{#1,#2}}\newcommand{\Ba}[1]{\mathbf{B}^{#1}_{6,1}}
\newcommand{\Me}{\mathbf{M}^{\epsilon}_{6}}\newcommand{\Lg}[2]{\mathbf{L}_{#1,#2}}
\newcommand{\Mg}[2]{\mathbf{M}_{#1,#2}}\newcommand{\bA}{\mathbf{A}}
\newcommand{\D}[2]{[\Delta_{#1#2}]}
\newcommand{\Zb}{Z^2_{BL}}\newcommand{\Hb}{H^2_{BL}}\newcommand{\Hm}{H^2_{M}}\newcommand{\Zm}{Z^2_{M}}

\begin{document}

\noindent\fbox{\parbox{\dimexpr\textwidth-2\fboxsep-2\fboxrule\relax}{\small
\textbf{Corrected version.} This is a corrected version of the paper H. Abdelwahab, A. J. Calder\'on, I. Kaygorodov,
\emph{The algebraic and geometric classification of nilpotent binary Lie algebras},
International Journal of Algebra and Computation, 29 (2019), no.~6, 1113--1129,
\doi{10.1142/S0218196719500437}.}}

\bigskip

\begin{center}
{\Large The algebraic and geometric classification of\\ nilpotent binary Lie algebras}

\bigskip

{\bf Hani Abdelwahab, Antonio Jes\'us Calder\'on and Ivan Kaygorodov}
\end{center}

\medskip

\noindent{\small\textbf{Abstract.} We give a complete algebraic classification of nilpotent binary Lie algebras of dimension at most 6 over an arbitrary field of characteristic not 2 and a complete geometric classification of nilpotent binary Lie algebras of dimension 6 over $\C$. As an application, we give an algebraic and geometric classification of nilpotent anticommutative $\mathfrak{CD}$-algebras of dimension at most 6.}

\bigskip


\bigskip


\bigskip

\tableofcontents

\section*{Introduction}
\addcontentsline{toc}{section}{Introduction}

There are many results related to both the algebraic and geometric classification of small-dimensional algebras in the varieties of Jordan, Lie, Leibniz, Zinbiel \emph{et al.} algebras; for algebraic results, see, for example, \cite{b2,b6,b9,b15,b17}; for geometric results, see, for example, \cite{b7,b10,b12,b13,b15}. Here, we give an algebraic and geometric classification of low-dimensional nilpotent binary Lie algebras.

Malcev defined binary Lie algebras as algebras such that \emph{every two-generated subalgebra is a Lie algebra}. Identities of the variety of binary Lie algebras were described by Gainov \cite{b5}. Note that every Lie algebra is a Malcev algebra and every Malcev algebra is a binary Lie algebra. The systematic study of Malcev and binary Lie algebras began with the work of Sagle \cite{b18}. Properties of binary Lie algebras were studied by Filippov \emph{et al.} \cite{b4,b6,b13,b17}. Another interesting subclass of binary Lie algebras is that of anticommutative $\mathfrak{CD}$-algebras. The idea of the definition of $\mathfrak{CD}$-algebras is to generalize a certain property of Jordan and Lie algebras (see Section~3).

Our method of classification of nilpotent binary Lie algebras is based on the calculation of central extensions of smaller nilpotent algebras from the same variety. The algebraic study of central extensions of Lie and non-Lie algebras has a long history \cite{b8,b20}. Skjelbred and Sund \cite{b20} used central extensions of Lie algebras for a classification of nilpotent Lie algebras. Using the method of \cite{b20}, all non-Lie central extensions of all 4-dimensional Malcev algebras \cite{b8}, all anticommutative central extensions of 3-dimensional anticommutative algebras \cite{b1} and some others were described. Also, all 4-dimensional nilpotent associative algebras, all 4-dimensional nilpotent Novikov algebras, all 4-dimensional nilpotent bicommutative algebras and some others were classified in a similar way \cite{b3,b11,b14}.

\section{The algebraic classification of binary Lie algebras}

\subsection{Definitions and notation}

Throughout the paper, $\F$ denotes a field of characteristic not 2 and the multiplication of an algebra is specified by giving only the nonzero products among the basis elements.

In an anticommutative algebra $(\bA,[-,-])$, we define the Jacobian $\J(x,y,z)$ of elements $x,y,z$ in $\bA$ in the following way:
\[
\J(x,y,z):=[[x,y],z]+[[y,z],x]+[[z,x],y].
\]
It is clear that the Jacobian $\J(x,y,z)$ is skew-symmetric in its arguments.

\begin{definition}
Let $(\bA,[-,-])$ be an anticommutative algebra. Then $(\bA,[-,-])$ is a:
\begin{itemize}
\item \emph{Lie algebra} if $\J(x,y,z)=0$ for all $x,y,z\in\bA$;
\item \emph{Malcev algebra} if
\begin{equation}\label{e11}
\J(x,y,[x,z])=[\J(x,y,z),x]\quad\text{for all }x,y,z\in\bA;
\end{equation}
\item \emph{binary Lie algebra} if
\begin{equation}\label{e12}
\J([x,y],x,y)=0\quad\text{for all }x,y\in\bA.
\end{equation}
\end{itemize}
\end{definition}

Every Lie algebra is a Malcev algebra and every Malcev algebra is a binary Lie algebra.

The linearization of the identity \eqref{e11} is
\[
[[w,y],[x,z]]=[[[w,x],y],z]+[[[x,y],z],w]+[[[y,z],w],x]+[[[z,w],x],y],\quad w,x,y,z\in\bA
\]
(see \cite{b16,b18}). Further, by \cite[p.~149]{b4} the linearization of the identity \eqref{e12} is
\begin{equation}\label{e13}
\J([x,y],z,t)+\J([x,t],z,y)+\J([z,y],x,t)+\J([z,t],x,y)=0,\quad x,y,z,t\in\bA.
\end{equation}

We define inductively $\bA^1=\bA$ and
\[
\bA^{n+1}=[\bA^n,\bA^1]+[\bA^{n-1},\bA^2]+\dots+[\bA^1,\bA^n].
\]
The algebra $\bA$ is \emph{nilpotent} if $\bA^n=0$ for some $n$.

We state the main results of the first part of this paper.

\begin{theorem}\label{t2}
Let $\mathcal{N}^6_{BL}(\F)$ denote the number of 6-dimensional nilpotent binary Lie algebras over $\F$. Then
\[
\mathcal{N}^6_{BL}(\F)=41+2|\F^*|+5|\F^*/(\F^*)^2|,
\]
where $|\F^*|$ and $|\F^*/(\F^*)^2|$ denote the, possibly infinite, cardinality of the multiplicative group $\F^*$ and the quotient group of $\F^*$ by the subgroup $(\F^*)^2=\{x^2:x\in\F^*\}$, respectively.
\end{theorem}

\begin{theorem}\label{t3}
Every 6-dimensional nilpotent non-Malcev binary Lie algebra over $\F$ is isomorphic to one of the following algebras:
\begin{align*}
\bullet\ \Ba{\alpha}&:\ [e_1,e_2]=e_4,\ [e_1,e_3]=e_5,\ [e_2,e_3]=\alpha e_6,\ [e_4,e_5]=e_6,\quad \alpha\in\F^*/(\F^*)^2\cup\{0\};\\
\bullet\ \bB62&:\ [e_1,e_2]=e_3,\ [e_3,e_4]=e_5,\ [e_4,e_5]=e_6;\\
\bullet\ \bB63&:\ [e_1,e_2]=e_3,\ [e_3,e_4]=e_5,\ [e_1,e_3]=e_6,\ [e_4,e_5]=e_6.
\end{align*}
Among these algebras there are precisely the following isomorphisms: $\Ba{\alpha}\cong\Ba{\beta}$ if and only if there is $\lambda\in\F^*$ such that $\beta=\lambda^2\alpha$.
\end{theorem}

The proofs of Theorems~\ref{t2} and~\ref{t3} follow from the algebraic classification of 6-dimensional nilpotent Lie algebras over $\F$ given in \cite{b2}, the algebraic classification of 6-dimensional nilpotent non-Lie Malcev algebras over $\F$ given in \cite{b9} and the algebraic classification of 6-dimensional non-Malcev binary Lie algebras over $\F$ (see Section~1.4).

\subsection{A method for the algebraic classification of nilpotent algebras}

Let $\bA$ be a binary Lie algebra over $\F$ and let $\mathbf V$ be a vector space over $\F$. Then the $\F$-linear space $\Zb(\bA,\mathbf V)$ is defined as the set of all skew-symmetric bilinear maps $\theta:\bA\times\bA\to\mathbf V$ such that
\begin{equation}\label{e14}
\theta([[x,y],x],y)=\theta([[x,y],y],x)
\end{equation}
for all $x,y\in\bA$. For a linear map $f$ from $\bA$ to $\mathbf V$, if we write $\delta f:\bA\times\bA\to\mathbf V$ by $\delta f(x,y)=f([x,y])$, then $\delta f\in\Zb(\bA,\mathbf V)$. We define $B^2(\bA,\mathbf V)=\{\theta=\delta f: f\in\mathrm{Hom}(\bA,\mathbf V)\}$. One can easily check that $B^2(\bA,\mathbf V)$ is a linear subspace of $\Zb(\bA,\mathbf V)$. We define the set $\Hb(\bA,\mathbf V)$ as the quotient space $\Zb(\bA,\mathbf V)/B^2(\bA,\mathbf V)$. The equivalence class of $\theta\in\Zb(\bA,\mathbf V)$ is denoted by $[\theta]\in\Hb(\bA,\mathbf V)$.

Let $\mathrm{Aut}(\bA)$ be the automorphism group of the binary Lie algebra $\bA$ and let $\varphi\in\mathrm{Aut}(\bA)$. For $\theta\in\Zb(\bA,\mathbf V)$ define $\varphi\theta(x,y)=\theta(\varphi(x),\varphi(y))$. Now $\varphi\theta\in\Zb(\bA,\mathbf V)$, so $\mathrm{Aut}(\bA)$ acts on $\Zb(\bA,\mathbf V)$. It is easy to verify that $B^2(\bA,\mathbf V)$ is invariant under the action of $\mathrm{Aut}(\bA)$ and so $\mathrm{Aut}(\bA)$ acts on $\Hb(\bA,\mathbf V)$.

Let $\bA$ be a binary Lie algebra of dimension $m<n$ over $\F$. Let $\mathbf V$ be an $\F$-vector space of dimension $n-m$. For every skew-symmetric bilinear map $\theta:\bA\times\bA\to\mathbf V$ define on the linear space $\bA_\theta:=\bA\oplus\mathbf V$ the bilinear product ``$[-,-]_{\bA_\theta}$'' by $[x+x',y+y']_{\bA_\theta}=[x,y]+\theta(x,y)$ for all $x,y\in\bA$, $x',y'\in\mathbf V$. It is easy to see that the algebra $\bA_\theta$ is a binary Lie algebra if and only if $\theta\in\Zb(\bA,\mathbf V)$. It is also clear that if $\bA$ is nilpotent, then so is $\bA_\theta$. If $\theta\in\Zb(\bA,\mathbf V)$, we call $\bA_\theta$ an $(n-m)$-\emph{dimensional central extension} of $\bA$ by $\mathbf V$. We also call $\theta^\perp=\{x\in\bA:\theta(x,\bA)=0\}$ the \emph{annihilator} of $\theta$.

We recall that the \emph{annihilator} of an algebra $\bA$ is defined as the ideal $\Ann(\bA)=\{x\in\bA:[x,\bA]=0\}$. It is easy to verify that $\Ann(\bA_\theta)=(\theta^\perp\cap\Ann(\bA))\oplus\mathbf V$.

As in \cite[Lemma~5]{b8}, we can also prove that every binary Lie algebra of dimension $n$ with $\Ann(\bA)\ne0$ can be expressed in the form $\bA_\theta$ for an $m$-dimensional binary Lie algebra $\bA$, where $m<n$, and a vector space $\mathbf V$ of dimension $n-m$ (here $\theta\in\Zb(\bA,\mathbf V)$).

To solve the isomorphism problem, we need to study the action of $\mathrm{Aut}(\bA)$ on $\Hb(\bA,\mathbf V)$. To do that, let us fix $e_1,\dots,e_s$ a basis of $\mathbf V$, and $\theta\in\Zb(\bA,\mathbf V)$. Then $\theta$ can be uniquely written as $\theta(x,y)=\sum_{i=1}^s\theta_i(x,y)e_i$, where $\theta_i\in\Zb(\bA,\F)$. Moreover, $\theta^\perp=\theta_1^\perp\cap\theta_2^\perp\cap\dots\cap\theta_s^\perp$. Further, $\theta\in B^2(\bA,\mathbf V)$ if and only if every $\theta_i\in B^2(\bA,\F)$.

Given a binary Lie algebra $\bA$, if $\bA=\mathbf B\oplus\F x$ is a direct sum of two ideals, then $\F x$ is an \emph{annihilator component} of $\bA$. It is not difficult to prove (see \cite[Lemma~13]{b8}) that given a binary Lie algebra $\bA_\theta$, if we write $\theta(x,y)=\sum_{i=1}^s\theta_i(x,y)e_i\in\Zb(\bA,\mathbf V)$ and $\theta^\perp\cap\Ann(\bA)=0$, then $\bA_\theta$ has an annihilator component if and only if $[\theta_1],[\theta_2],\dots,[\theta_s]$ are linearly dependent in $\Hb(\bA,\F)$.

Let $\bA$ be a binary Lie algebra over $\F$ and let $\mathbf V$ be a vector space over $\F$. The \emph{Grassmannian} $G_k(\mathbf V)$ is the set of all $k$-dimensional linear subspaces of $\mathbf V$. Let $G_s(\Hb(\bA,\F))$ be the Grassmannian of subspaces of dimension $s$ in $\Hb(\bA,\F)$. There is a natural action of $\mathrm{Aut}(\bA)$ on $G_s(\Hb(\bA,\F))$. Let $\varphi\in\mathrm{Aut}(\bA)$. For $\mathbf W=\langle[\theta_1],\dots,[\theta_s]\rangle\in G_s(\Hb(\bA,\F))$ define $\varphi\mathbf W=\langle[\varphi\theta_1],\dots,[\varphi\theta_s]\rangle$. Then $\varphi\mathbf W\in G_s(\Hb(\bA,\F))$. We denote the orbit of $\mathbf W\in G_s(\Hb(\bA,\F))$ under the action of $\mathrm{Aut}(\bA)$ by $\mathrm{Orb}(\mathbf W)$. Let
\[
\mathbf W_1=\langle[\theta_1],\dots,[\theta_s]\rangle,\quad \mathbf W_2=\langle[\vartheta_1],\dots,[\vartheta_s]\rangle\in G_s(\Hb(\bA,\F)).
\]
If $\mathbf W_1=\mathbf W_2$, then $\bigcap_{i=1}^s\theta_i^\perp\cap\Ann(\bA)=\bigcap_{i=1}^s\vartheta_i^\perp\cap\Ann(\bA)$. So
\[
\mathbf T_s(\bA)=\Big\{\langle[\theta_1],\dots,[\theta_s]\rangle\in G_s(\Hb(\bA,\F)):\ \bigcap_{i=1}^s\theta_i^\perp\cap\Ann(\bA)=0\Big\}
\]
is well defined and stable under the action of $\mathrm{Aut}(\bA)$.

Now, let $\mathbf V$ be an $s$-dimensional linear space and let us denote by $\mathbf E(\bA,\mathbf V)$ the set of all binary Lie algebras without annihilator components which are $s$-dimensional central extensions of $\bA$ by $\mathbf V$ and have $s$-dimensional annihilator. Let
$\mathbf E(\bA,\mathbf V)=\{\bA_\theta:\theta(x,y)=\sum_{i=1}^s\theta_i(x,y)e_i,\ \langle[\theta_1],\dots,[\theta_s]\rangle\in\mathbf T_s(\bA)\}$.

\begin{lemma}\label{l4}
Let $\bA_\theta,\bA_\vartheta\in\mathbf E(\bA,\mathbf V)$. Suppose that $\theta(x,y)=\sum_{i=1}^s\theta_i(x,y)e_i$ and $\vartheta(x,y)=\sum_{i=1}^s\vartheta_i(x,y)e_i$. Then the binary Lie algebras $\bA_\theta$ and $\bA_\vartheta$ are isomorphic if and only if
\[
\mathrm{Orb}\langle[\theta_1],\dots,[\theta_s]\rangle=\mathrm{Orb}\langle[\vartheta_1],\dots,[\vartheta_s]\rangle.
\]
\end{lemma}
\begin{proof}
The proof is similar to \cite[Lemma~17]{b8}.
\end{proof}

Hence, there exists a one-to-one correspondence between the set of $\mathrm{Aut}(\bA)$-orbits on $\mathbf T_s(\bA)$ and the set of isomorphism classes of $\mathbf E(\bA,\mathbf V)$. Consequently, we have a procedure that, given the (nilpotent) binary Lie algebras $\bA'$ of dimension $n-s$, constructs all of the (nilpotent) binary Lie algebras $\bA$ of dimension $n$ with no annihilator components and with $s$-dimensional annihilator. This procedure is the following:
\begin{enumerate}
\item[(1)] For a given (nilpotent) binary Lie algebra $\bA'$ of dimension $n-s$, determine $\mathbf T_s(\bA')$ and $\mathrm{Aut}(\bA')$.
\item[(2)] Determine the set of $\mathrm{Aut}(\bA')$-orbits on $\mathbf T_s(\bA')$.
\item[(3)] For each orbit, construct the binary Lie algebra corresponding to a representative of it.
\end{enumerate}

The above method gives all (Malcev and non-Malcev) binary Lie algebras. But we also are interested in developing this method in such a way that it only gives non-Malcev binary Lie algebras. Clearly, every central extension of a non-Malcev binary Lie algebra is non-Malcev. So we only have to study the central extensions of Malcev algebras. Let $\mathbf M$ be a Malcev algebra and $\theta\in\Zb(\mathbf M,\F)$. Then $\mathbf M_\theta$ is a Malcev algebra if and only if
\[
\theta([w,y],[x,z])=\theta([[w,x],y],z)+\theta([[x,y],z],w)+\theta([[y,z],w],x)+\theta([[z,w],x],y)
\]
for all $x,y,z,w\in\mathbf M$. Define a subspace $\Zm(\mathbf M,\F)$ of $\Zb(\mathbf M,\F)$ as the set of all $\theta\in\Zb(\mathbf M,\F)$ satisfying this identity. Define $\Hm(\mathbf M,\F)=\Zm(\mathbf M,\F)/B^2(\mathbf M,\F)$. Therefore, $\Hm(\mathbf M,\F)$ is a subspace of $\Hb(\mathbf M,\F)$. Define
\[
\begin{aligned}
\mathbf R_s(\mathbf M)&=\{\mathbf W\in\mathbf T_s(\mathbf M):\mathbf W\in G_s(\Hm(\mathbf M,\F))\},\\
\mathbf U_s(\mathbf M)&=\{\mathbf W\in\mathbf T_s(\mathbf M):\mathbf W\notin G_s(\Hm(\mathbf M,\F))\}.
\end{aligned}
\]
Then $\mathbf T_s(\mathbf M)=\mathbf R_s(\mathbf M)\,\dot\cup\,\mathbf U_s(\mathbf M)$. The sets $\mathbf R_s(\mathbf M)$ and $\mathbf U_s(\mathbf M)$ are stable under the action of $\mathrm{Aut}(\mathbf M)$. Thus, the binary Lie algebras corresponding to the representatives of $\mathrm{Aut}(\mathbf M)$-orbits on $\mathbf R_s(\mathbf M)$ are Malcev algebras while those corresponding to the representatives of $\mathrm{Aut}(\mathbf M)$-orbits on $\mathbf U_s(\mathbf M)$ are not. Hence, given those binary Lie algebras $\bA'$ of dimension $n-s$, we may construct all non-Malcev algebras $\bA$ of dimension $n$ with $s$-dimensional annihilator which have no annihilator components, in the following way:
\begin{enumerate}
\item[(1)] For a given binary Lie algebra $\bA'$ of dimension $n-s$, if $\bA'$ is non-Malcev then apply the procedure described above.
\item[(2)] Otherwise, do the following:
\begin{enumerate}
\item[(a)] Determine $\mathbf U_s(\bA')$ and $\mathrm{Aut}(\bA')$.
\item[(b)] Determine the set of $\mathrm{Aut}(\bA')$-orbits on $\mathbf U_s(\bA')$.
\item[(c)] For each orbit, construct the binary Lie algebra corresponding to a representative of it.
\end{enumerate}
\end{enumerate}

Finally, let us introduce notation. Let $\bA$ be a binary Lie algebra with basis $e_1,e_2,\dots,e_n$. By $\Delta_{ij}$, we denote the skew-symmetric bilinear form $\Delta_{ij}:\bA\times\bA\to\F$ with $\Delta_{ij}(e_i,e_j)=-\Delta_{ij}(e_j,e_i)=1$ and $\Delta_{ij}(e_l,e_m)=0$ if $\{i,j\}\ne\{l,m\}$. Then $\{\Delta_{ij}:1\le i<j\le n\}$ is a basis for the linear space of skew-symmetric bilinear forms on $\bA$. Every $\theta\in\Zb(\bA,\F)$ can be uniquely written as $\theta=\sum_{1\le i<j\le n}c_{ij}\Delta_{ij}$, where $c_{ij}\in\F$. Further, let $\theta=\sum c_{ij}\Delta_{ij}$ be a skew-symmetric bilinear form on $\bA$. Then $\theta\in\Zb(\bA,\F)$ if and only if the $c_{ij}$'s satisfy property \eqref{e14}. We can decide this by computer. Note that property \eqref{e14} is not linear in $x,y$; it is better to linearize it. For that we have the following lemma.

\begin{lemma}\label{l5}
Let $\bA$ be a binary Lie algebra and $\theta\in\Zb(\bA,\F)$. Then
\begin{equation}\label{e15}
\psi_\theta([x,y],z,t)+\psi_\theta([x,t],z,y)+\psi_\theta([z,y],x,t)+\psi_\theta([z,t],x,y)=0,
\end{equation}
where $\psi_\theta(x,y,z):=\theta([x,y],z)+\theta([y,z],x)+\theta([z,x],y)$.
\end{lemma}
\begin{proof}
Let $\theta\in\Zb(\bA,\F)$. Then $\bA_\theta$ is a binary Lie algebra. We denote the Jacobian of elements $x,y,z$ in $\bA_\theta$ by $\J_{\bA_\theta}(x,y,z)$. Now consider $x,y,z,t\in\bA$. Then
\begin{align*}
\J_{\bA_\theta}([x,y],z,t)&=\J([x,y],z,t)+\psi_\theta([x,y],z,t),\\
\J_{\bA_\theta}([x,t],z,y)&=\J([x,t],z,y)+\psi_\theta([x,t],z,y),\\
\J_{\bA_\theta}([z,y],x,t)&=\J([z,y],x,t)+\psi_\theta([z,y],x,t),\\
\J_{\bA_\theta}([z,t],x,y)&=\J([z,t],x,y)+\psi_\theta([z,t],x,y).
\end{align*}
By the identity \eqref{e13}, the sum of the left-hand sides is zero, as well as the sum of $\J([x,y],z,t)+\J([x,t],z,y)+\J([z,y],x,t)+\J([z,t],x,y)$. We deduce \eqref{e15}.
\end{proof}
Note that Eq.~\eqref{e14} can be obtained from Eq.~\eqref{e15} by taking $z=x$, $t=y$ since the characteristic of $\F$ is not 2.

\subsection{Nilpotent binary Lie algebras of dimensions at most 5}

In this section, the classification of nilpotent binary Lie algebras of dimension at most 5 is given. Throughout the paper, we use some notational conventions: $\mathbf L_{i,j}$ is the $j$th nilpotent Lie algebra of dimension $i$; $\mathbf M_{i,j}$ is the $j$th nilpotent non-Lie Malcev algebra of dimension $i$; $\mathbf B_{i,j}$ is the $j$th nilpotent non-Malcev binary Lie algebra of dimension $i$; and the basis elements of an algebra of dimension $i$ are denoted by $e_1,e_2,\dots,e_i$. It is known from \cite{b17} that every nilpotent binary Lie algebra of dimension at most 4 over $\F$ is a nilpotent Lie algebra and thus $\Hb(\bA,\F)=\Hm(\bA,\F)$ for every nilpotent binary Lie algebra $\bA$ of dimension at most 3, since otherwise we would have a 4-dimensional nilpotent binary Lie algebra which is neither a Lie algebra nor a Malcev algebra.

\begin{theorem}\label{t6}
Every nilpotent binary Lie algebra of dimension at most 4 is isomorphic to one of the pairwise nonisomorphic algebras in Table~1.
\end{theorem}

\begin{center}
\textbf{Table 1.} Nilpotent binary Lie algebras of dimension at most 4.\\[2mm]
\renewcommand{\arraystretch}{1.25}
\begin{tabular}{|l|l|l|l|}
\hline
$\bA$ & Multiplication & $\Hb(\bA,\F)=\Hm(\bA,\F)$ & $\Ann(\bA)$\\\hline
$\Lg11$ & --- & $0$ & $\Lg11$\\\hline
$\Lg21$ & --- & $\langle\D12\rangle$ & $\Lg21$\\\hline
$\Lg31$ & --- & $\langle\D12,\D13,\D23\rangle$ & $\Lg31$\\\hline
$\Lg32$ & $[e_1,e_2]=e_3$ & $\langle\D13,\D23\rangle$ & $\langle e_3\rangle$\\\hline
$\Lg41$ & --- & $\langle\D12,\D13,\D14,\D23,\D24,\D34\rangle$ & $\Lg41$\\\hline
$\Lg42$ & $[e_1,e_2]=e_3$ & $\langle\D13,\D14,\D23,\D24,\D34\rangle$ & $\langle e_3,e_4\rangle$\\\hline
$\Lg43$ & $[e_1,e_2]=e_3,\ [e_1,e_3]=e_4$ & $\langle\D14,\D23\rangle$ & $\langle e_4\rangle$\\\hline
\end{tabular}
\end{center}

\begin{lemma}\label{l7}
Let $\bA$ be an $n$-dimensional nilpotent binary Lie algebra.
\begin{enumerate}
\item[(1)] If $n\le4$, then $\Hb(\bA,\F)=\Hm(\bA,\F)$ and so $\mathbf U_s(\bA)=\emptyset$ for $s\ge1$.
\item[(2)] If $\bA$ is non-Malcev, then $\dim\Ann(\bA)\le n-5$.
\item[(3)] If $n=5$, then $\bA$ is a Malcev algebra.
\end{enumerate}
\end{lemma}
\begin{proof}
(1) It follows from Table~1.

(2) Suppose to the contrary that $\dim\Ann(\bA)>n-5$. Then $\dim\bA/\Ann(\bA)\le4$ and so $\bA/\Ann(\bA)$ is a Lie algebra. Further, $\bA$ can be viewed as a central extension $\mathbf L_\theta$ of $\mathbf L:=\bA/\Ann(\bA)$ by a vector space $\mathbf V$ of dimension $s=\dim\Ann(\bA)$, with $\theta\in\Zb(\mathbf L,\mathbf V)$. Since $\dim\mathbf L\le4$, we have $\Hb(\mathbf L,\F)=\Hm(\mathbf L,\F)$ by (1); since both $\Zm(\mathbf L,\F)\subseteq\Zb(\mathbf L,\F)$ contain $B^2(\mathbf L,\F)$, this means $\Zb(\mathbf L,\F)=\Zm(\mathbf L,\F)$. Hence every component $\theta_i$ of $\theta$ lies in $\Zm(\mathbf L,\F)$, i.e.\ $\mathbf U_s(\mathbf L)=\emptyset$. Therefore $\bA$ is a Malcev algebra, which is a contradiction.

(3) It follows from (1). Also, since $\bA$ is nilpotent, $\dim\Ann(\bA)\ge1$ and therefore, by (2), $\bA$ is a Malcev algebra.
\end{proof}

\begin{theorem}\label{t8}
Every 5-dimensional nilpotent binary Lie algebra is a Malcev algebra and isomorphic to one of the pairwise nonisomorphic algebras in Table~2. Moreover,
$\Ann(\Lg51)=\Lg51$, $\Ann(\Lg52)=\langle e_3,e_4,e_5\rangle$, $\Ann(\Lg53)=\langle e_4,e_5\rangle$, $\Ann(\Lg58)=\langle e_4,e_5\rangle$, $\Ann(\Lg59)=\langle e_4,e_5\rangle$, $\Ann(\Lg5i)=\langle e_5\rangle$ for $i=4,5,6,7$, $\Ann(\Mg51)=\langle e_5\rangle$, and
\[
\mathbf U_1(\Lg58)\ne\emptyset,\quad \mathbf U_1(\Mg51)\ne\emptyset,\quad \mathbf U_1(\Lg5i)=\emptyset,\ i\in\{1,2,3,4,5,6,7,9\}.
\]
\end{theorem}

\begin{center}\begin{minipage}{\linewidth}\centering
\textbf{Table 2.} Nilpotent binary Lie algebras of dimension 5.\\[2mm]
\footnotesize\setlength{\tabcolsep}{3pt}\renewcommand{\arraystretch}{1.3}
\begin{tabular}{|l|p{4.8cm}|p{4.5cm}|p{3.0cm}|}
\hline
$\bA$ & Multiplication & $\Hm(\bA,\F)$ & $\Hb(\bA,\F)$\\\hline
$\Lg51$ & --- & $\Hb(\Lg41,\F)\oplus\langle\D15,\allowbreak\D25,\allowbreak\D35,\allowbreak\D45\rangle$ & $\Hm(\Lg51,\F)$\\\hline
$\Lg52$ & $[e_1,e_2]=e_3$ & $\Hb(\Lg42,\F)\oplus\langle\D15,\allowbreak\D25,\allowbreak\D35,\allowbreak\D45\rangle$ & $\Hm(\Lg52,\F)$\\\hline
$\Lg53$ & $\begin{array}{l}[e_1,e_2]=e_3,\\{}[e_1,e_3]=e_4\end{array}$ & $\Hb(\Lg43,\F)\oplus\langle\D15,\allowbreak\D25,\allowbreak\D35\rangle$ & $\Hm(\Lg53,\F)$\\\hline
$\Lg54$ & $[e_1,e_2]=e_5,\ [e_3,e_4]=e_5$ & $\langle\D13,\allowbreak\D14,\allowbreak\D23,\allowbreak\D24,\allowbreak\D34,\allowbreak\D15,\allowbreak\D25,\allowbreak\D35,\allowbreak\D45\rangle$ & $\Hm(\Lg54,\F)$\\\hline
$\Lg55$ & $\begin{array}{l}[e_1,e_2]=e_3,\\{}[e_1,e_3]=e_5,\ [e_2,e_4]=e_5\end{array}$ & $\langle\D13,\allowbreak\D14,\allowbreak\D23,\allowbreak\D34,\allowbreak\D15\rangle$ & $\Hm(\Lg55,\F)$\\\hline
$\Lg56$ & $\begin{array}{l}[e_1,e_2]=e_3,\ [e_1,e_3]=e_4,\\{}[e_1,e_4]=e_5,\ [e_2,e_3]=e_5\end{array}$ & $\langle\D14,\ \D15+\D24,\ \D25-\D34\rangle$ & $\Hm(\Lg56,\F)$\\\hline
$\Lg57$ & $\begin{array}{l}[e_1,e_2]=e_3,\ [e_1,e_3]=e_4,\\{}[e_1,e_4]=e_5\end{array}$ & $\langle\D15,\allowbreak\D23,\allowbreak\D25-\D34\rangle$ & $\Hm(\Lg57,\F)$\\\hline
$\Lg58$ & $[e_1,e_2]=e_4,\ [e_1,e_3]=e_5$ & $\langle\D14,\allowbreak\D15,\allowbreak\D23,\allowbreak\D24,\allowbreak\D25,\allowbreak\D34,\allowbreak\D35\rangle$ & $\Hm(\Lg58,\F)\oplus\langle\D45\rangle$\\\hline
$\Lg59$ & $\begin{array}{l}[e_1,e_2]=e_3,\ [e_1,e_3]=e_4,\\{}[e_2,e_3]=e_5\end{array}$ & $\langle\D14,\ \D15+\D24,\ \D25\rangle$ & $\Hm(\Lg59,\F)$\\\hline
$\Mg51$ & $[e_1,e_2]=e_3,\ [e_3,e_4]=e_5$ & $\langle\D13,\allowbreak\D14,\allowbreak\D23,\allowbreak\D24\rangle$ & $\Hm(\Mg51,\F)\oplus\langle\D45\rangle$\\\hline
\end{tabular}
\end{minipage}
\end{center}

\subsection{Nilpotent binary Lie algebras of dimension 6}

In this section, we give a complete classification of all 6-dimensional nilpotent binary Lie algebras over $\F$. Nilpotent Lie algebras of dimension 6 over $\F$ were classified in \cite{b2}; nilpotent non-Lie Malcev algebras were classified in \cite{b9}. Therefore, we only classify nilpotent binary Lie algebras which are not Malcev algebras. Every nilpotent binary Lie algebra of dimension 5 is a Malcev algebra. Therefore, 6-dimensional nilpotent non-Malcev binary Lie algebras with annihilator components do not exist. Next, we classify 6-dimensional nilpotent non-Malcev binary Lie algebras without any annihilator component. By Theorem~\ref{t8}, for a 5-dimensional nilpotent binary Lie algebra $\bA$, $\mathbf U_1(\bA)\ne\emptyset$ if and only if $\bA\cong\Lg58$ or $\bA\cong\Mg51$.

\subsubsection{The binary Lie algebras corresponding to the representatives of $\mathrm{Aut}(\Lg58)$-orbits on $\mathbf U_1(\Lg58)$}

The automorphism group of $\Lg58$ consists of the invertible matrices of the form
\[
\varphi=\begin{pmatrix}
a_{11}&0&0&0&0\\
a_{21}&a_{22}&a_{23}&0&0\\
a_{31}&a_{32}&a_{33}&0&0\\
a_{41}&a_{42}&a_{43}&a_{11}a_{22}&a_{11}a_{23}\\
a_{51}&a_{52}&a_{53}&a_{11}a_{32}&a_{11}a_{33}
\end{pmatrix}
\]
(the $j$th column is the coordinate column of $\varphi(e_j)$). Choose an arbitrary subspace $\mathbf W\in\mathbf U_1(\Lg58)$. From Table~2, such a subspace is spanned by
\[
[\theta]=C_{14}\D14+C_{15}\D15+C_{23}\D23+C_{24}\D24+C_{25}\D25+C_{34}\D34+C_{35}\D35+C_{45}\D45
\]
such that $C_{45}\ne0$. Let $\varphi=(a_{ij})\in\mathrm{Aut}(\Lg58)$. Write $[\varphi\theta]=C'_{14}\D14+C'_{15}\D15+C'_{23}\D23+C'_{24}\D24+C'_{25}\D25+C'_{34}\D34+C'_{35}\D35+C'_{45}\D45$. Then
\begin{align*}
C'_{14}&=a_{11}(C_{14}a_{11}a_{22}+C_{15}a_{11}a_{32}+C_{24}a_{21}a_{22}+C_{25}a_{21}a_{32}\\
&\quad+C_{34}a_{22}a_{31}+C_{35}a_{31}a_{32}-C_{45}a_{22}a_{51}+C_{45}a_{32}a_{41}),\\
C'_{15}&=a_{11}(C_{14}a_{11}a_{23}+C_{15}a_{11}a_{33}+C_{24}a_{21}a_{23}+C_{25}a_{21}a_{33}\\
&\quad+C_{34}a_{31}a_{23}+C_{35}a_{31}a_{33}-C_{45}a_{23}a_{51}+C_{45}a_{41}a_{33}),\\
C'_{23}&=C_{23}(a_{22}a_{33}-a_{23}a_{32})+C_{24}(a_{22}a_{43}-a_{23}a_{42})+C_{25}(a_{22}a_{53}-a_{23}a_{52})\\
&\quad+C_{34}(a_{32}a_{43}-a_{33}a_{42})+C_{35}(a_{32}a_{53}-a_{33}a_{52})+C_{45}(a_{42}a_{53}-a_{43}a_{52}),\\
C'_{24}&=a_{11}(C_{24}a_{22}^2+C_{35}a_{32}^2+C_{25}a_{22}a_{32}+C_{34}a_{22}a_{32}-C_{45}a_{22}a_{52}+C_{45}a_{32}a_{42}),\\
C'_{25}&=a_{11}(C_{24}a_{22}a_{23}+C_{25}a_{22}a_{33}+C_{34}a_{23}a_{32}+C_{35}a_{32}a_{33}-C_{45}a_{23}a_{52}+C_{45}a_{33}a_{42}),\\
C'_{34}&=a_{11}(C_{24}a_{22}a_{23}+C_{25}a_{23}a_{32}+C_{34}a_{22}a_{33}+C_{35}a_{32}a_{33}-C_{45}a_{22}a_{53}+C_{45}a_{32}a_{43}),\\
C'_{35}&=a_{11}(C_{24}a_{23}^2+C_{35}a_{33}^2+C_{25}a_{23}a_{33}+C_{34}a_{23}a_{33}-C_{45}a_{23}a_{53}+C_{45}a_{33}a_{43}),\\
C'_{45}&=a_{11}^2(a_{22}a_{33}-a_{23}a_{32})C_{45}.
\end{align*}
Set $\delta=C_{23}C_{45}-C_{24}C_{35}+C_{25}C_{34}$ and $\delta'=C'_{23}C'_{45}-C'_{24}C'_{35}+C'_{25}C'_{34}$. Easy computations show that $\delta'=a_{11}^2(a_{22}a_{33}-a_{23}a_{32})^2\delta$. Thus $\mathrm{Orb}(\langle[\theta]:\delta\ne0\rangle)\cap\mathrm{Orb}(\langle[\theta]:\delta=0\rangle)=\emptyset$ and hence $\mathrm{Aut}(\Lg58)$ has at least two orbits on $\mathbf U_1(\Lg58)$.

Let $\varphi$ be the automorphism with $a_{11}=a_{22}=a_{44}=1$, $a_{33}=a_{55}=C_{45}^{-1}$,
$a_{41}=-C_{45}^{-1}C_{15}$, $a_{42}=-C_{45}^{-1}C_{25}$, $a_{43}=-C_{45}^{-2}C_{35}$, $a_{51}=C_{45}^{-1}C_{14}$, $a_{52}=C_{45}^{-1}C_{24}$, $a_{53}=C_{45}^{-2}C_{34}$, and all other $a_{ij}=0$:
\[
\varphi=\begin{pmatrix}
1&0&0&0&0\\
0&1&0&0&0\\
0&0&C_{45}^{-1}&0&0\\
-C_{45}^{-1}C_{15}&-C_{45}^{-1}C_{25}&-C_{45}^{-2}C_{35}&1&0\\
C_{45}^{-1}C_{14}&C_{45}^{-1}C_{24}&C_{45}^{-2}C_{34}&0&C_{45}^{-1}
\end{pmatrix}.
\]
Then $\varphi\mathbf W=\langle C_{45}^{-2}\delta\,\D23+\D45\rangle$. Set $\alpha=C_{45}^{-2}\delta$.

$\bullet$ \textbf{Case 1.} $\delta\ne0$. Then $\alpha\ne0$ and so we get the representatives $\mathbf W_\alpha=\langle\alpha\D23+\D45\rangle$, $\alpha\in\F^*$. We claim that $\mathrm{Orb}(\mathbf W_\alpha)=\mathrm{Orb}(\mathbf W_\beta)$ if and only if there is $\lambda\in\F^*$ such that $\beta=\lambda^2\alpha$. Hence, the number of possible orbits among such representatives is $|\F^*/(\F^*)^2|$. To see this, suppose that $\mathrm{Orb}(\mathbf W_\alpha)=\mathrm{Orb}(\mathbf W_\beta)$. Then there exist $\varphi=(a_{ij})\in\mathrm{Aut}(\Lg58)$ and $\lambda\in\F^*$ such that $\varphi(\beta\D23+\D45)=\lambda(\alpha\D23+\D45)$. Consequently, we obtain the following polynomial equations:
\begin{gather*}
a_{11}(a_{32}a_{41}-a_{22}a_{51})=0,\quad a_{11}(a_{41}a_{33}-a_{23}a_{51})=0,\quad a_{11}(a_{32}a_{42}-a_{22}a_{52})=0,\\
a_{11}(a_{33}a_{42}-a_{23}a_{52})=0,\quad a_{11}(a_{32}a_{43}-a_{22}a_{53})=0,\quad a_{11}(a_{33}a_{43}-a_{23}a_{53})=0,\\
a_{11}^2(a_{22}a_{33}-a_{23}a_{32})=\lambda,\quad a_{42}a_{53}-a_{43}a_{52}+\beta(a_{22}a_{33}-a_{23}a_{32})=\lambda\alpha.
\end{gather*}
Since $\det\varphi\ne0$ if and only if $a_{11}(a_{22}a_{33}-a_{23}a_{32})\ne0$, the first six equations give $a_{41}=a_{42}=a_{43}=a_{51}=a_{52}=a_{53}=0$, and hence $a_{42}a_{53}-a_{43}a_{52}=0$. We obtain from the last two equations that $\beta=a_{11}^2\alpha$. Conversely, suppose that $\beta=\lambda^2\alpha$ for some $\lambda\in\F^*$. Let $\varphi$ be the diagonal matrix with the entries $(\lambda,1,1,\lambda,\lambda)$ on the diagonal. Then $\varphi\mathbf W_\beta=\langle\beta\D23+\lambda^2\D45\rangle=\langle\lambda^2(\alpha\D23+\D45)\rangle=\mathbf W_\alpha$. This completes the proof of the claim. Hence, we get the following algebras:
\[
\Ba{\alpha\ne0}:\ [e_1,e_2]=e_4,\ [e_1,e_3]=e_5,\ [e_2,e_3]=\alpha e_6,\ [e_4,e_5]=e_6.
\]
Moreover, $\Ba{\alpha\ne0}$ and $\Ba{\beta\ne0}$ are isomorphic if and only if there is $\lambda\in\F^*$ such that $\beta=\lambda^2\alpha$. So the number of nonisomorphic algebras in the family $\Ba{\alpha\ne0}$ is $|\F^*/(\F^*)^2|$.

$\bullet$ \textbf{Case 2.} $\delta=0$. Then $\varphi\mathbf W=\langle\D45\rangle$. So we get the algebra
\[
\Ba{0}:\ [e_1,e_2]=e_4,\ [e_1,e_3]=e_5,\ [e_4,e_5]=e_6.
\]

\subsubsection{The binary Lie algebras corresponding to the representatives of $\mathrm{Aut}(\Mg51)$-orbits on $\mathbf U_1(\Mg51)$}

The automorphism group of $\Mg51$ consists of the invertible matrices of the form
\[
\varphi=\begin{pmatrix}
a_{11}&a_{12}&0&0&0\\
a_{21}&a_{22}&0&0&0\\
0&0&a_{11}a_{22}-a_{12}a_{21}&a_{34}&0\\
0&0&0&a_{44}&0\\
a_{51}&a_{52}&0&a_{54}&a_{44}(a_{11}a_{22}-a_{12}a_{21})
\end{pmatrix}.
\]
Choose an arbitrary subspace $\mathbf W\in\mathbf U_1(\Mg51)$. From Table~2, such a subspace is spanned by
\[
[\theta]=C_{13}\D13+C_{14}\D14+C_{23}\D23+C_{24}\D24+C_{45}\D45,
\]
where $C_{45}\ne0$. Let $\varphi=(a_{ij})\in\mathrm{Aut}(\Mg51)$ and $D=a_{11}a_{22}-a_{12}a_{21}$. Write $[\varphi\theta]=C'_{13}\D13+C'_{14}\D14+C'_{23}\D23+C'_{24}\D24+C'_{45}\D45$. Then
\begin{align*}
C'_{13}&=(C_{13}a_{11}+C_{23}a_{21})D,\\
C'_{14}&=C_{13}a_{11}a_{34}+C_{14}a_{11}a_{44}+C_{23}a_{21}a_{34}+C_{24}a_{21}a_{44}-C_{45}a_{51}a_{44},\\
C'_{23}&=(C_{13}a_{12}+C_{23}a_{22})D,\\
C'_{24}&=C_{13}a_{12}a_{34}+C_{14}a_{12}a_{44}+C_{23}a_{22}a_{34}+C_{24}a_{22}a_{44}-C_{45}a_{52}a_{44},\\
C'_{45}&=C_{45}a_{44}^2D.
\end{align*}
It is clear that if $C_{13}=C_{23}=0$ then $C'_{13}=C'_{23}=0$. Hence,
\[
\mathrm{Orb}(\langle[\theta]:(C_{13},C_{23})=(0,0)\rangle)\cap\mathrm{Orb}(\langle[\theta]:(C_{13},C_{23})\ne(0,0)\rangle)=\emptyset,
\]
so $\mathrm{Aut}(\Mg51)$ has at least two orbits on $\mathbf U_1(\Mg51)$.

$\bullet$ \textbf{Case 1.} $(C_{13},C_{23})=(0,0)$. Let $\varphi$ be the automorphism with $a_{11}=C_{45}^{-1}$, $a_{22}=a_{44}=1$, $a_{33}=a_{55}=C_{45}^{-1}$, $a_{51}=C_{45}^{-2}C_{14}$, $a_{52}=C_{45}^{-1}C_{24}$ and all other $a_{ij}=0$. Then $\varphi\mathbf W=\langle\D45\rangle$. So we get the algebra
\[
\bB62:\ [e_1,e_2]=e_3,\ [e_3,e_4]=e_5,\ [e_4,e_5]=e_6.
\]

$\bullet$ \textbf{Case 2.} $(C_{13},C_{23})\ne(0,0)$. Suppose first that $C_{13}\ne0$. Let $\varphi$ be the automorphism with $a_{11}=C_{45}C_{13}^{-1}$, $a_{12}=-C_{23}C_{13}^{-1}$, $a_{22}=a_{44}=1$, $a_{33}=a_{55}=a_{11}$, $a_{51}=C_{14}C_{13}^{-1}$, $a_{52}=C_{45}^{-1}(C_{14}a_{12}+C_{24})$ and all other $a_{ij}=0$. Then $\varphi\mathbf W=\langle\D13+\D45\rangle$. Assume now that $C_{13}=0$. Then $C_{23}\ne0$. Let $\varphi$ be the automorphism with $a_{12}=a_{44}=1$, $a_{21}=C_{45}C_{23}^{-1}$, $a_{33}=a_{55}=-a_{21}$, $a_{51}=C_{23}^{-1}C_{24}$, $a_{52}=C_{45}^{-1}C_{14}$ and all other $a_{ij}=0$. Then again $\varphi\mathbf W=\langle\D13+\D45\rangle$. This shows that if $(C_{13},C_{23})\ne(0,0)$, then we get only one algebra:
\[
\bB63:\ [e_1,e_2]=e_3,\ [e_3,e_4]=e_5,\ [e_1,e_3]=e_6,\ [e_4,e_5]=e_6.
\]

\section{The geometric classification of nilpotent binary Lie algebras}

\subsection{Definitions and notation}

Given an $n$-dimensional complex vector space $\mathbf V$, the set $\mathrm{Hom}(\mathbf V\otimes\mathbf V,\mathbf V)\cong\mathbf V^*\otimes\mathbf V^*\otimes\mathbf V$ is a vector space of dimension $n^3$. This space has the structure of the affine variety $\C^{n^3}$. Fix a basis $e_1,\dots,e_n$ of $\mathbf V$. Every $\mu\in\mathrm{Hom}(\mathbf V\otimes\mathbf V,\mathbf V)$ is determined by the $n^3$ structure constants $c^k_{i,j}\in\C$ such that $\mu(e_i\otimes e_j)=\sum_{k=1}^nc^k_{i,j}e_k$. A subset of $\mathrm{Hom}(\mathbf V\otimes\mathbf V,\mathbf V)$ is \emph{Zariski-closed} if it can be defined by a set of polynomial equations in the variables $c^k_{i,j}$ ($1\le i,j,k\le n$).

Let $T$ be a set of polynomial identities. All algebra structures on $\mathbf V$ satisfying polynomial identities from $T$ form a Zariski-closed subset of the variety $\mathrm{Hom}(\mathbf V\otimes\mathbf V,\mathbf V)$. We denote this subset by $\mathbb L(T)$. The general linear group $\mathrm{GL}(\mathbf V)$ acts on $\mathbb L(T)$ by conjugation:
\[
(g*\mu)(x\otimes y)=g\mu(g^{-1}x\otimes g^{-1}y)
\]
for $x,y\in\mathbf V$, $\mu\in\mathbb L(T)\subset\mathrm{Hom}(\mathbf V\otimes\mathbf V,\mathbf V)$ and $g\in\mathrm{GL}(\mathbf V)$. Thus, $\mathbb L(T)$ is decomposed into $\mathrm{GL}(\mathbf V)$-orbits that correspond to the isomorphism classes of algebras. Let $O(\mu)$ denote the orbit of $\mu\in\mathbb L(T)$ under the action of $\mathrm{GL}(\mathbf V)$ and let $\overline{O(\mu)}$ denote the Zariski closure of $O(\mu)$.

Let $\bA$ and $\mathbf B$ be two $n$-dimensional algebras satisfying identities from $T$ and $\mu,\lambda\in\mathbb L(T)$ represent $\bA$ and $\mathbf B$, respectively. We say that $\bA$ degenerates to $\mathbf B$, and write $\bA\to\mathbf B$, if $\lambda\in\overline{O(\mu)}$. In this case $\overline{O(\lambda)}\subset\overline{O(\mu)}$. Hence, the definition of a degeneration does not depend on the choice of $\mu$ or $\lambda$. If $\bA\not\cong\mathbf B$, then the assertion $\bA\to\mathbf B$ is a \emph{proper degeneration}. We write $\bA\not\to\mathbf B$ if $\lambda\notin\overline{O(\mu)}$.

Let $\bA$ be represented by $\mu\in\mathbb L(T)$. Then $\bA$ is \emph{rigid} in $\mathbb L(T)$ if $O(\mu)$ is an open subset of $\mathbb L(T)$. Recall that a subset of a variety is irreducible if it cannot be represented as a union of two nontrivial closed subsets. A maximal irreducible closed subset of a variety is an \emph{irreducible component}. It is well known that every affine variety can be represented as a finite union of its irreducible components in a unique way. The algebra $\bA$ is rigid in $\mathbb L(T)$ if and only if $\overline{O(\mu)}$ is an irreducible component of $\mathbb L(T)$.

\subsection{Degenerations of algebras}

We use the methods applied to Lie algebras in \cite{b7}. If $\bA\to\mathbf B$ and $\bA\not\cong\mathbf B$, then $\dim\Der(\bA)<\dim\Der(\mathbf B)$, where $\Der(\bA)$ is the Lie algebra of derivations of $\bA$. We will compute the dimensions of algebras of derivations and will check the assertion $\bA\to\mathbf B$ only for such $\bA$ and $\mathbf B$ that $\dim\Der(\bA)<\dim\Der(\mathbf B)$.

To prove degenerations, we construct families of matrices parametrized by $t$. Namely, let $\bA$ and $\mathbf B$ be two algebras represented by the structures $\mu$ and $\lambda$ from $\mathbb L(T)$, respectively. Let $e_1,\dots,e_n$ be a basis of $\mathbf V$ and let $c^k_{i,j}$ ($1\le i,j,k\le n$) be the structure constants of $\lambda$ in this basis. If there exist $a^j_i(t)\in\C$ ($1\le i,j\le n$, $t\in\C^*$) such that $E^t_i=\sum_{j=1}^na^j_i(t)e_j$ ($1\le i\le n$) form a basis of $\mathbf V$ for every $t\in\C^*$, and the structure constants of $\mu$ in the basis $E^t_1,\dots,E^t_n$ are polynomials $c^k_{i,j}(t)\in\C[t]$ such that $c^k_{i,j}(0)=c^k_{i,j}$, then $\bA\to\mathbf B$. In this case $E^t_1,\dots,E^t_n$ is a \emph{parametrized basis} for $\bA\to\mathbf B$.

To prove non-degenerations, besides the inequality for $\dim\Der$ we use the following simple invariant (the closedness of the identities of $\mathbb L(T)$ is used in the same way for the Lie, Malcev and binary Lie algebras).

\begin{lemmaA}\label{lA}
Let $\bA\to\mathbf B$ be a degeneration of $n$-dimensional anticommutative algebras. If $[\bA^2,\bA^2]=0$ (that is, $\bA^2=[\bA,\bA]$ is an abelian subalgebra), then $[\mathbf B^2,\mathbf B^2]=0$.
\end{lemmaA}
\begin{proof}
Let $\mu$ and $\lambda$ represent $\bA$ and $\mathbf B$ and let $r=\dim\bA^2$. Since the orbit $O(\mu)$ is a constructible set, its Zariski closure coincides with its closure in the Euclidean topology, so there are $g_k\in\mathrm{GL}(\mathbf V)$ with $\mu_k:=g_k*\mu\to\lambda$. Put $W_k=\mathrm{Im}\,\mu_k=g_k(\bA^2)$; this is an $r$-dimensional subspace and $\mu_k(W_k,W_k)=0$. Passing to a subsequence, we may assume that $W_k\to W$ in the Grassmannian $G_r(\mathbf V)$. Every value $\lambda(x,y)=\lim\mu_k(x,y)$ is a limit of vectors from $W_k$, so $\mathbf B^2\subseteq W$. If $u,v\in W$, choose $u_k,v_k\in W_k$ with $u_k\to u$, $v_k\to v$; then $0=\mu_k(u_k,v_k)\to\lambda(u,v)$, so $\lambda(W,W)=0$. Hence $[\mathbf B^2,\mathbf B^2]\subseteq\lambda(W,W)=0$.
\end{proof}

\subsection{The geometric classification of 6-dimensional nilpotent binary Lie algebras}

The geometric classification of 6-dimensional nilpotent binary Lie algebras is based on the description of all degenerations of 6-dimensional Malcev algebras. Thanks to \cite{b13}, the variety of 6-dimensional nilpotent Malcev algebras has only two irreducible components defined by the following algebras:
\begin{align*}
\mathbf g_6&:\ [e_1,e_2]=e_3,\ [e_1,e_3]=e_4,\ [e_1,e_4]=e_5,\ [e_2,e_3]=e_5,\ [e_2,e_5]=e_6,\ [e_3,e_4]=-e_6;\\
\Me&:\ [e_1,e_2]=e_3,\ [e_1,e_3]=e_5,\ [e_1,e_5]=e_6,\ [e_2,e_4]=\epsilon e_5,\ [e_3,e_4]=e_6.
\end{align*}
(The first is a Lie algebra, the second is a family of non-Lie Malcev algebras.) The main result of the present section is the following theorem.

\begin{theorem}\label{t9}
The variety of 6-dimensional nilpotent binary Lie algebras over $\C$ has three irreducible components defined by the rigid algebras $\bB63$, $\mathbf g_6$ and $\Ba1$.
\end{theorem}

\begin{proof}
Note that $\dim\Der(\bB63)=\dim\Der(\mathbf g_6)=8$, $\dim\Der(\Ba1)=9$, and $\dim\Der(\bB62)=\dim\Der(\Ba0)=\dim\Der(\Me)=10$. There is no degeneration between $\bB63$ and $\mathbf g_6$ since they have equal dimensions of the derivation algebras. We construct the following degenerations.

$\bullet$ The parametrized basis formed by
$E^t_1=te_1-ite_4$, $E^t_2=e_2-\epsilon e_3+(\epsilon^2-\epsilon)ie_5$, $E^t_3=te_3-i\epsilon te_5+(\epsilon^2-2\epsilon)te_6$, $E^t_4=-t^2e_1$, $E^t_5=it^2e_5+(1-\epsilon)t^2e_6$, $E^t_6=t^3e_6$
gives the degeneration $\bB63\to\Me$.

$\bullet$ The parametrized basis formed by $E^t_1=te_1$, $E^t_2=t^{-1}e_2$, $E^t_3=e_3$, $E^t_4=e_4$, $E^t_5=e_5$, $E^t_6=e_6$ gives the degeneration $\bB63\to\bB62$.

$\bullet$ The parametrized basis formed by $E^t_1=t^{-1}e_1$, $E^t_2=e_2$, $E^t_3=e_3$, $E^t_4=t^{-1}e_4$, $E^t_5=t^{-1}e_5$, $E^t_6=t^{-2}e_6$ gives the degeneration $\Ba1\to\Ba0$.

The first three degenerations, the result of \cite{b13} on the Malcev algebras and Theorem~\ref{t3} (over $\C$ all algebras $\Ba\alpha$, $\alpha\ne0$, are isomorphic to $\Ba1$) show that every 6-dimensional nilpotent binary Lie algebra lies in $\overline{O(\bB63)}\cup\overline{O(\mathbf g_6)}\cup\overline{O(\Ba1)}$: every Malcev algebra lies in the closure of $O(\mathbf g_6)$ or of the family $\Me$, and the latter is contained in $\overline{O(\bB63)}$; the non-Malcev algebras are $\bB63$, $\bB62\in\overline{O(\bB63)}$, $\Ba1$ and $\Ba0\in\overline{O(\Ba1)}$.

It remains to show that none of the three closures $\overline{O(\bB63)}$, $\overline{O(\mathbf g_6)}$, $\overline{O(\Ba1)}$ is contained in another one. We have $\bB63\not\to\mathbf g_6$ and $\mathbf g_6\not\to\bB63$ by the equality of the dimensions of the derivation algebras. Since $\mathbf g_6$ is a Lie algebra and the Jacobi identity is closed under degenerations, $\mathbf g_6$ degenerates only to Lie algebras; hence $\mathbf g_6\not\to\Ba1$. Further, $\Ba1\to\bB63$ and $\Ba1\to\mathbf g_6$ are impossible because $\dim\Der(\Ba1)=9>8$. Finally, $\bB63\not\to\Ba1$: indeed, $\bB63^2=\langle e_3,e_5,e_6\rangle$ and $[e_3,e_5]=[e_3,e_6]=[e_5,e_6]=0$, so $[\bB63^2,\bB63^2]=0$, while $(\Ba1)^2=\langle e_4,e_5,e_6\rangle$ and $[e_4,e_5]=e_6\ne0$; by Lemma~A, $\bB63\not\to\Ba1$. Therefore, the three orbit closures are irreducible, closed and pairwise non-contained, and so they are the irreducible components of the variety.
\end{proof}

\section{Application: classification of anticommutative $\mathfrak{CD}$-algebras}

The class of nonassociative $\mathfrak{CD}$-algebras is defined by a certain property of Jordan and Lie algebras: \emph{every commutator of two multiplication operators is a derivation}. Namely, an algebra $\bA$ is a $\mathfrak{CD}$-algebra if and only if
\[
T_xT_y-T_yT_x\in\Der(\bA),\quad\text{for all }x,y\in\bA,\ T_z\in\{R_z,L_z\}.
\]
It is easy to see that the class of $\mathfrak{CD}$-algebras is defined by three identities of degree 4. In the case of commutative and anticommutative $\mathfrak{CD}$-algebras, there is only one identity:
\begin{multline}\label{e31}
[[[x,y],a],b]-[[[x,y],b],a]\\=[[[x,a],b],y]-[[[x,b],a],y]+[x,[[y,a],b]]-[x,[[y,b],a]].
\end{multline}
If we set $a=y$ and $b=x$ in Eq.~\eqref{e31}, then $[[[x,y],y],x]=[[[x,y],x],y]$. We conclude that every anticommutative $\mathfrak{CD}$-algebra is a binary Lie algebra. So the variety of anticommutative $\mathfrak{CD}$-algebras is between Lie and binary Lie algebras. It is clear that if an algebra $\bA$ satisfies $\bA^4=0$, then $\bA$ is a $\mathfrak{CD}$-algebra. So every 5-dimensional nilpotent binary Lie algebra is a $\mathfrak{CD}$-algebra. By some easy checking of identity \eqref{e31} for all 6-dimensional nilpotent binary Lie algebras, we obtain the following result.

\begin{theorem}\label{t10}
Let $\bA$ be a 6-dimensional nilpotent anticommutative $\mathfrak{CD}$-algebra over $\F$. Then $\bA$ is isomorphic to a Malcev algebra or to $\Ba\alpha$. Every 6-dimensional nilpotent Malcev algebra over $\F$ is a $\mathfrak{CD}$-algebra.
\end{theorem}

As a corollary of Theorem~\ref{t10}, we obtain the following result.

\begin{theorem}\label{t11}
The variety of 6-dimensional nilpotent anticommutative $\mathfrak{CD}$-algebras over $\C$ has three irreducible components defined by the family of algebras $\Me$ and the rigid algebras $\Ba1$, $\mathbf g_6$.
\end{theorem}

\begin{proof}
Using the algebraic classification of 6-dimensional nilpotent anticommutative $\mathfrak{CD}$-algebras (Theorem~\ref{t10}), the geometric classification of 6-dimensional nilpotent binary Lie algebras (Theorem~\ref{t9}), and the description of all degenerations of 6-dimensional nilpotent Malcev algebras \cite{b13}, we obtain that $\mathbf g_6$ is rigid in the variety of binary Lie algebras, and hence in its closed subvariety of $\mathfrak{CD}$-algebras. By Theorem~\ref{t9}, $\Ba1$ is also rigid in the variety of binary Lie algebras, and hence in the variety of $\mathfrak{CD}$-algebras, since $\Ba1$ is a $\mathfrak{CD}$-algebra. By Theorem~\ref{t10}, the variety of 6-dimensional nilpotent anticommutative $\mathfrak{CD}$-algebras is the union of the variety of Malcev algebras and of $\overline{O(\Ba1)}$ (recall that $\Ba\alpha\cong\Ba1$ for $\alpha\ne0$ and $\Ba1\to\Ba0$). Irreducible components defined by the family of Malcev algebras $\Me$ and the rigid algebra $\Ba1$ have the same dimension (namely 27) and they are different: the first one consists of Malcev algebras only, while $\Ba1$ is not a Malcev algebra; and no member of the family $\Me$ lies in $\overline{O(\Ba1)}$ and vice versa, since $\overline{O(\Ba1)}\setminus O(\Ba1)$ has dimension at most 26.
\end{proof}

\section*{Appendix A. Differences from the published version}
\addcontentsline{toc}{section}{Appendix A. Differences from the published version}

\noindent\textbf{Theorem~\ref{t9}:} statement changed; the variety has three irreducible components (the rigid algebras $\bB63$, $\mathbf g_6$ and $\Ba1$), not two.  

\noindent\textbf{Theorem~\ref{t11}:} the statement is unchanged; the proof is changed.

\end{document}